\documentclass[reqno,12pt, a4] {amsart}
\usepackage[usenames,dvipsnames,svgnames,table]{xcolor}

\usepackage[active]{srcltx}
\usepackage{pb-diagram}

\usepackage{mathrsfs}
\usepackage{amsmath}
\usepackage{amssymb}
\usepackage{amscd}
\usepackage{amsthm}
\usepackage[latin1]{inputenc}
\usepackage{graphics}
\usepackage{varioref}

\labelformat{enumi}{(#1)}

\newtheorem{teo}[equation]{Theorem}

\newtheorem{cor}[equation]{Corollary}
\newtheorem{lemma}[equation]{Lemma}

\newtheoremstyle{named}{}{}{\itshape}{}{\bfseries}{.}{.5em}{\thmnote{#3}#1}
\theoremstyle{named}
\newtheorem*{namedtheorem}{}

\theoremstyle{definition}
\newtheorem{Example}{Example}

\newcommand{\Keler}             {K\"{a}hler }

\newcommand{\cd}{\cdot}
\renewcommand{\setminus}{-}
\newcommand{\om}{\omega}
\newcommand{\tmu}{\tilde{\mu}}
\newcommand{\tmup}{\tilde{\mu}_\liep}
\newcommand{\mutp}{\tilde{\mu}_\liep}
\newcommand{\omt}{\tilde{\omega}}

\newcommand{\xbp} {X_{||\beta||^2}^{\beta+}}

\renewcommand{\phi}{\varphi}
\newcommand{\cinf}{C^\infty}

\newcommand{\ra}{\rightarrow}
\newcommand{\lra}{\longrightarrow}
\newcommand{\C}{\mathbb{C}}
\newcommand{\R}{\mathbb{R}}

\newcommand{\restr}[1]          {\vert_{#1}}

\newcommand{\Ad}{\operatorname{Ad}}

\newcommand{\demi}{\frac{1}{2}}

\newcommand{\ga}{\gamma}

\newcommand{\meno}{^{-1}}

\newcommand{\PP}{\mathbb{P}}

\newcommand{\enf}{\emph}

\newcommand{\liu}{\mathfrak{u}}
\newcommand{\lia}{\mathfrak{a}}

\newcommand{\mut}{\tilde{\mu}}
\newcommand{\lieg}{\mathfrak{g}}

\newcommand{\liep}{\mathfrak{p}}
\newcommand{\lieb}{\mathfrak{b}}

\newcommand{\rrr}{{||\beta||^2}}
\newcommand{\liet}{\mathfrak{t}}

\newcommand{\chern}{\operatorname{c}}

\newcommand{\vacuo}{\emptyset}
\newcommand{\sx}{\langle}                   % scalar product
\newcommand{\xs}{\rangle}
\newcommand{\Crit}{\operatorname{Crit}}     % set of critical points
\newcommand{\mup}{\mu_\liep}
\newcommand{\mua}{\mu_\lia}

\newcommand{\scalo} {\sx \ , \ \xs}

\newcommand{\mupbe}{\mu_{\liep^\beta}}

\begin{document}

\author{Leonardo Biliotti \and Alessandro Ghigi \and Peter Heinzner}
\title{A remark on the gradient map}

\address{Universit\`{a} di Parma} \email{leonardo.biliotti@unipr.it}
\address{Universit\`a di Milano Bicocca}
\email{alessandro.ghigi@unimib.it} \address{Ruhr Universit\"at Bochum}
\email{peter.heinzner@rub.de}

\thanks{The
  first two authors were partially supported by a grant of Max-Plank
  Institute f\"ur Mathematik, Bonn, and by FIRB 2012 MIUR ``Geometria
  differenziale e teoria geometrica delle
  funzioni''. % grant RBFR12W1AQ003.
  The second author was partially supported also by PRIN 2009 MIUR
  ''Moduli, strutture geo\-me\-tri\-che e loro applicazioni''.  The
  third author was partially supported by DFG-priority program SPP
  1388 (Darstellungstheorie).}
\begin{abstract}
  For a Hamiltonian action of a compact group $U$ of isometries on a
  compact K\"ahler manifold $Z$ and a compatible subgroup $G$ of
  $U^\C$, we prove that for any closed $G$--invariant subset $Y\subset
  Z$ the image of the gradient map $\mup(Y)$ is independent of the
  choice of the invariant K\"ahler form $\om$ in its cohomology class
  $[\om]$.
\end{abstract}

\maketitle

\section{Introduction}

Let $(Z, \om)$ be a compact \Keler manifold and let $U$ be a compact
connected semisimple Lie group such that $U^\C$ acts holomorphically
on $Z$, $U$ preserves $\om$ and there is a momentum map  $\mu: Z \ra
\liu^* $. Let $G\subset U^\C$ be a \emph{compatible} subgroup. By this
we mean a subgroup which is compatible with the Cartan
involution $\Theta$ of $U^\C$ which defines $U$, i.e. if $\liep = \lieg \cap
i \liu$ and $K=U \cap G$, then $G= K \cd \exp \liep$. Let $ \mu_\liep : Z \ra \liep $ be the associated \emph{gradient map}
(see \cite{heinzner-schwarz-stoetzel, heinzner-stoetzel-global} or section \ref{background}).

In this note we prove the following.

\begin{teo} \label{image-0} Let $Y\subset Z$ be a closed $G$-stable
  subset. Then up to translation the set $\mup (Y)$ is independent of
  the choice of the invariant K\"ahler form $\om$ in the cohomology
  class $[\om]$.
\end{teo}

Since $Z$ is compact and G is compatible there is a stratification of
$Z$ analogous to the Kirwan stratification, see
%Let $X\subset Z$ be a compact $G$-invariant real submanifold. In
\cite{heinzner-schwarz-stoetzel}. This gives a stratification of any closed $G$--invariant subset $Y$ of $Z$, by intersecting the strata in $Z$ with $Y$.
%it is shown that there is a stratification of $X$ analogous to the Kirwan stratification.
It
follows from Theorem \ref{image-0} that when the momentum map is
properly normalized (see Lemma \ref{defmut}) this stratification does
not depend on the choice of $\om$ in its cohomology class.

When $Z$ is a projective manifold and $\om$ is the pull-back of a
Fubini-Study form via an equivariant embedding of $Z$ in $\PP^N$,
Kirwan \cite[\S 12]{kirwan} proved that the stratification in terms of
a properly normalized $\mu$ can be defined purely in terms of
algebraic geometry.  %By the same reasoning one can prove more
%generally that the stratification depends only on the class $[\om]$,
%as soon as this class is integral.
In the present note we give a proof of this fact for a general compact
K\"ahler manifold $Z$ in the more general setting of \emph{gradient}
maps for actions of compatible subgroups on closed $G$--invariant
subsets of $Z$.

Another consequence of the above is the following. Assume that $Z$ is
a projective manifold and that $[\om]$ is an integral class. Let
$Y\subset Z$ be a closed $G$-invariant real semi-algebraic subset
whose real algebraic Zariski closure is irreducible. Let $\lia \subset
\liep$ be a maximal subalgebra and let $\lia_+$ be a closed Weyl
chamber in $\lia$.  Then $A(Y)_+:=\mup(Y)\cap\lia_+ $ is convex (see
\cite{heinzner-schuetzdeller}, which deals with the case when $\om$ is
the restriction of a Fubini-Study metric).

{\bfseries \noindent{Acknowledgements.}}  The first two authors are
grateful to the Fakul\-t\"at f\"ur Mathematik of Ruhr-Universit\"at
Bochum for the wonderful hospitality during several visits.  They also
wish to thank the Max-Planck Institut f\"ur Mathematik, Bonn for
excellent conditions provided during their visit at this institution,
where part of this paper was written.

\section{Background}\label{background}

Let $(Z, \om)$ be a compact \Keler manifold and let $U$ be a compact
Lie group. Assume that $U$ acts on $Z$ by holomorphic \Keler
  isometries.  Since $Z$ is compact the $U$-action extends to a
  holomorphic action of the complexified group $U^\C$.  Assume also
that there is a momentum map $\mu: Z \ra \liu^*\cong \liu$, where
$\liu^*$ is identified with $\liu$ using a fixed $U$-invariant scalar
product on $\liu$ that we denote by $\scalo$. We also denote by
$\scalo$ the scalar product on $i\liu$ such that multiplication by $i$
is an isometry of $\liu$ onto $i\liu$.  If $\xi \in \liu$ we denote by
$\xi_Z$ the fundamental vector field on $Z$ and we let $\mu^\xi \in
\cinf(Z)$ be the function $\mu^\xi(z) := \sx \mu(z),\xi\xs$.  That
$\mu$ is the momentum map means that it is $U$-equivariant and that
$d\mu^\xi = i_{\xi_Z} \om$.

For a closed subgroup $G\subset U^\C$ let $K:=G\cap U$ and $\liep:=
\lieg \cap i\liu$.  The group $G$ is called \enf{compatible} if $G =K \cd
\exp \liep$ \cite{heinzner-schwarz-stoetzel,
  heinzner-stoetzel-global}.
% One can define a $\R$-bilinear form $B$ on $\liu^\C$ by imposing
% $B(\liu, i \liu)=0$, $B= -\scalo$ on $\liu$ and $B= \scalo$ on
% $i\liu$.  Then $B$ is $\Ad U^\C$-invariant and
% nondegenerate. \label{def-B}
In the following we fix a compatible subgroup $G \subset U^\C$.  If
$z \in Z$, let $\mup (z) \in \liep$ denote $-i$ times the component of
$\mu(z)$ in the direction of $i\liep$.  In other words we require that
$\sx \mup (z) , \beta \xs = - \sx \mu(z) , i\beta\xs$ for any $\beta
\in \liep$.  The map
\begin{gather*}
  \mu_\liep : Z \ra \liep
\end{gather*}
is called the \emph{gradient map} (see \cite{heinzner-schwarz-Cartan})
or \emph{restricted momentum map}. Let $\mup^\beta \in \cinf(Z)$ be
the function $ \mup^\beta(z) = \sx \mup(z) , \beta\xs =
\mu^{-i\beta}(z)$.  Let $(\ , \ )$ be the \Keler metric associated to
$\om$, i.e. $(v, w) = \om (v, Jw)$. Then $\beta_Z$ is the gradient of
$\mup^\beta$ with respect to $(\ , \ )$.
% If $X \subset Z$ is a $G$-invariant real (embedded) submanifold,
% then $\beta_X$ is the gradient of $\mup^\beta \restr{X}$ with
% respect to the induced Riemannian structure on $X$.

\begin{Example}
  (1) For any compact subgroup $K \subset U$, both $K$ and its
  complexification $G=K^\C$ are compatible. In particular $G=U^\C$ is
  a compatible subgroup.  (2) If $G$ is a real form of $U^\C$, then
  $G$ is compatible.  (3) For any $\xi\in i\liu$, the subgroup
  $G=\exp(\R \xi)$ is compatible.
% More generally, if $\lia \subset i \liu$ is a Lie subalgebra, then it is commutative and $G=\exp(i\lia + \lia)$ is a compatible subgroup of $U^\C$.
% \end{enumerate}
\end{Example}

Next we recall the Stratification Theorem for actions of compatible
subgroups.  Given a maximal subalgebra $\lia \subset \liep$ and a Weyl
chamber $\lia^+ \subset \lia$ define
\begin{gather*}
  \eta_\liep : X \ra \R \qquad  \eta_\liep (x):= \demi ||\mu_\liep(x)||^2 \\
  C_\liep :=\Crit (\eta_\liep) \qquad \mathcal{B}_\liep := \mu_\liep
  (C_\liep) \qquad \mathcal{B}_\liep^+ := \mathcal{B}_\liep \cap
  \lia^+\\
  X(\mu)=\{x\in X:\, \overline{G\cd x} \cap \mup\meno(0) \neq
  \vacuo\}
\end{gather*}
where $X$ is a compact $G$-invariant subset of $Z$. Points lying in
$X(\mu)$ are called \emph{semistable}.  Using semistability and the
function $\eta_\liep$ one can define a stratification of $X$ in the
following way, see \cite{kirwan} and \cite{heinzner-schwarz-stoetzel}.
For $\beta \in \mathcal{B}^{+}_\liep$ %and $r=||\beta ||^2$
 set
\begin{gather*}
  X_\rrr : = \{ x\in X: \overline{\exp(\R \beta ) \cd x } \cap (\mu^\beta)\meno(\rrr)\} \\
  X^\beta := \{x\in X: \beta_X(x) = 0\} \\
  X_\rrr^\beta : = X^\beta \cap X_\rrr \\
%  X^{\beta+} : = \{ x \in X: \lim_{t \to -\infty } \exp(t\beta)\cd x \text{ exists}\}\\
  X_\rrr^{\beta+} : = \{ x\in X_\rrr:
\lim_{t \to -\infty } \exp(t\beta)\cdot x \text{ exists and  it lies in }
  X^\beta_\rrr
  \}\\
 G^{\beta+} :=\{g\in G : \text{the limit } \lim_{t\to -\infty}\exp(t\beta) g
  \exp(-t\beta) \text{ exists in } G \} .
\end{gather*}
 Set also
\begin{gather*}
  G^\beta = \{g\in G: \Ad g (\beta ) = \beta\}\qquad \liep^\beta : =
  \{\xi \in \liep: [\xi, \beta ] =0 \}.
\end{gather*}
The group $G^\beta = K^\beta \cd \exp(\liep^\beta)$ is a compatible subgroup of $U^\C$ and
the set $\xbp$ is $G^{\beta+}$-invariant. Denote by $\mupbe$ the
composition of $\mup$ with the orthogonal projection $\liep \ra
\liep^\beta$.  Then $\mupbe$ is a gradient map for the
$G^\beta$-action on $\xbp$. We set $\widehat{\mupbe} : = \mupbe -
\beta$. Since $\beta$ lies in the center of $\lieg^\beta$ and since  $G^\beta$ is
a compatible subgroup of $(U^\beta)^\C= (U^\C)^\beta$, it is a
gradient map too.  We let $S^{\beta+}$ denote the set of
$G^\beta$-semistable points in $\xbp$ with respect to
$\widehat{\mupbe}$, i.e.
\begin{gather*}
  S^{\beta+} : = \{x\in \xbp: \overline{G^\beta \cd x} \cap \mupbe
  \meno (\beta) \neq \vacuo\}.
\end{gather*}
The set $S^{\beta+}$ coincides with the set of semistable points
  of the group $G^\beta$ in $\xbp$ after shifting.  By definition the
  $\beta$-stratum is given by $S_\beta : = G \cd S^{\beta+}$.
\begin{namedtheorem}
  [Stratification Theorem] \textnormal{(See
  {\cite[Thm. 7.3]{heinzner-schwarz-stoetzel}})} Assume that
%  $\mup\restr{X} : X \ra \liep$ is proper
$X$ is a compact $G$-invariant subset of $Z$. Then $
\mathcal{B}_\liep^+ $ is finite and
  \begin{gather*}
    X = \bigsqcup _{\beta \in \mathcal{B}_\liep^+ } S_\beta.
  \end{gather*}
  Moreover
  \begin{gather*}
    \overline{S_\beta} \subset S_\beta \cup \bigcup_{||\ga|| > ||\beta|| } S_\ga.
  \end{gather*}
%
%
%   $ S_\beta$ is contained in $ S_\beta$ together
%   with strata $S_\ga$ corresponding to $\ga \in \mathcal{B}_\liep^+$
%   with $ ||\ga|| > ||\beta|| $.
\end{namedtheorem}

\section{Proof of Theorem \ref{image-0}}
For a $U$-invariant function $f$ on $Z$ we set
\begin{gather*}
  \omt := \om + dd^c f
\end{gather*}
where $ d^c f : = -2J^* df$.  Since $Z$ is compact and $U$ acts by
holomorphic transformations, any $U$-invariant K\"ahler form $\omt$ in
the K\"ahler class $[\om]$ can be written in this way.  Since
pluriharmonic functions on $Z$ are constant, the function $f$ is
unique up to a constant.
\begin{lemma}\label{defmut}
  If $\mu: Z \ra \liu$ is a momentum map for the $U$-action on $Z$
  with respect to $\om$, then the function $\mut: Z \ra \liu$ defined
  by
  \begin{gather} \label{eq:defmut} \mut^\xi: = \mu^\xi - d^cf(\xi_Z)
  \end{gather}
  is a momentum map for the $U$-action on $Z$ with respect to $\omt$.
\end{lemma}
\begin{proof}

That $\mut$ is a momentum map follows from Cartan formula using that
  $L_{\xi_Z} d^c f = d^c L_{\xi_Z} f = 0$.  This in turn follows from
  the assumption that the action of $U$ is holomorphic and $f$ is
  $U$-invariant.
\end{proof}

% \begin{proof}
%   Since $\xi_Z$ is a holomorphic vector field $L_{\xi_Z} d^c f = d^c
%   L_{\xi_Z} f$. Since $f$ is $U$-invariant $L_{\xi_Z} d^c f =0$,
%   hence $- d i_{\xi_Z} d^c f = i_{\xi_Z}d d^c f$.  So
%   \begin{gather*}
%     d \mut^\xi = d\mu^\xi - d i_{\xi_Z} d^c f = i_{\xi_Z}\om +
%     i_{\xi_Z}d d^c f = i_{\xi_Z}\omt.
%   \end{gather*}
% \end{proof}

% That $\mut$ is a momentum map follows immediately from Cartan
% formula using that $L_{\xi_Z} d^c f = d^c L_{\xi_Z} f = 0$.  This in
% turn follows from the assumption that the action of $U$ is
% holomorphic and $f$ is $U$-invariant.  The choice of $\mut$ as a
% momentum map appears e.g. in
% \cite[p. 72]{heinzner-huckleberry-Inventiones}.

A more precise version of Theorem \ref{image-0} is the following.
\begin{teo} \label{image} For any closed $G$-stable subset $Y \subset
  Z$ we have $\mup (Y) = \tmup(Y)$.
\end{teo}
\begin{proof}
  Let $\lia \subset \liep$ be a maximal subalgebra and set $A:= \exp
  \lia$. The group $A$ is a compatible subgroup. Let $\mua: Z \ra
  \lia$ be the restricted gradient map.  Any connected subgroup
  $B\subset A$ is compatible. Given such a $B$, set $Z^{(B)} := \{
  z\in Z: A_z = B\}$.  A connected component $S$ of $Z^{(B)}$ will be
  called an $A$-stratum of type $\lieb$. For a given $S$ let $C$
  denote the connected component of $Z^B$ containing $S$. Then $C$ is
  a complex submanifold of $Z$ and the
  % Since the stabilizer $A_z$ is always compatible, it is connected.
  % Therefore from compactness of $Z$ and the
  Slice Theorem (see Theorem 14.10 and 14.21 in
  \cite{heinzner-schwarz-Cartan} or Theorem 2.2 in
  \cite{heinzner-schuetzdeller}) applied to the $A$-action on $C$
  shows that $S$ is open and dense in $C$.
%
  % we get that $Z$ is a finite disjoint union of the $Z^{(B)}$.
  %   % Let $P$ be the image of $\mua (Z)$. Since $\mua (Z)$ is a
  %   % projection of the moment polytope for the $T$-action, $P$ is a
  %   % convex polytope whose vertices are images of $A^\C$-fixed
  %   % points.
  % An $\lia$-stratum $S$ is by definition a connected component of
  % $Z^{(B)}$ for some connected $B\subset A$. We claim that
  % $\overline{S}$ is a connected component of the fixed point set
  % $Z^B=\{p\in Z:\, gp=p$ for all $g \in B \}$, which is a smooth
  % complex submanifold of $Z$.  From the inclusion $S \subset
  % Z^{(B)}$, it follows that $\overline{S} \subset Z^B$. Set
  % $A':=A/B$. $A'$ acts on
  % $Z^B$. %The action of $A'$ on $Z^B$ is Hamiltionian.
  % Let $C$ be the connected component of $Z^B$ containing $\overline
  % S$.  It follows from the Slice Theorem that the principal isotropy
  % of the $A'$-action on $C$ is trivial and the set $\{y\in C: A'_y =
  % \{1\} \} = S\cap C$ is open and dense in $C$.  Hence
  % $\overline{S}$ is a connected component of $Z^B$ and therefore it
  % is a compact complex submanifold of $Z$.

  Let $A^c$ be the Zariski closure of $A$ in $U^\C$. The group $A^c$
  is a compatible subgroup of $U^\C$, $A^c \cap U=T$ is a torus and
  $A^c = T \exp(i\liet)$, where $ \liet$ denotes the Lie algebra of
  $T$. Moreover $\overline S$ is $A^c$-stable \cite[Lemma 3.3 (1)]
  {heinzner-schuetzdeller}.  Denote by $\mu_{\mathfrak t}:Z \lra
  \mathfrak t$ the momentum map obtained by projecting $\mu:Z \lra
  \mathfrak u$ to $\mathfrak t$, and denote by $\Pi : i \liet \ra
  \lia$ the orthogonal projection.  Then $\mu_{\mathfrak a} = \Pi
  \circ i \mu_\liet$ and $\mu_{\lia}(\overline S)=\Pi (i\mu_{\mathfrak
    t}(\overline S))$.  By the convexity theorem of
  Atiyah-Guillemin-Sternberg $\mu_{\mathfrak t}(\overline S)$ is a
  convex polytope and its vertices are images of points fixed by
  $A^c$. It follows that $\mu_{\mathfrak a}(\overline S )$ is a convex
  polytope as well.  Since $\Pi$ is linear, any vertex of
  $\mu_{\mathfrak a}(\overline S )$ is the projection of at least one
  vertex of $i\mu_{\mathfrak t}(\overline S)$.  Therefore
  $\mu_{\mathfrak a}(\overline S )$ is the convex hull of
  $\mu_\lia(\overline{S}^A)$.
%
%
  % $A^c$-fixed points, in particular of $A$-fixed points.
%
%
  % $\mu_{\mathfrak a}(\overline S )$ is a convex hull of the image of
  % the $A$-fixed points.
  Now we use Lemma \ref{defmut}: if $x \in \overline{S}^A$, then
  $\xi_Z(x) = 0$, so $\mut^\xi(x) = \mu^\xi(x)$, for any $\xi \in
  \lia$. Therefore $\mut_\lia(x) = \mu_\lia(x)$ for every $A$-fixed
  point $x$.  It follows that both $\mu_{\mathfrak a} (\overline S)$
  and the affine subspace spanned by $\mua(S)$ do not depend on the
  choice of the \Keler form $\om$.

  % Now recall the following construction from
  % \cite[p. 1124]{heinzner-schuetzdeller}:

  Let $\Sigma$ be the collection of affine hyperplanes of $\lia$ that
  are affine hulls of $\mu_\lia(\overline{S})$ for some $A$-stratum
  $S$.  Set $P:= \mua(Z)$ and
  \begin{gather*}
    P_0:= P \setminus \bigcup_{H \in \Sigma } P\cap H.
  \end{gather*}
  (See \cite{heinzner-schuetzdeller}). The set $P_0$ is an open subset
  of $\lia$. Let $C(P_0)$ denote the set of its connected
  components. This is a finite set.  For $\ga \in C(P_0)$ let $P(\ga)$
  be the closure of the connected component $\ga$.  Then $P(\ga)$ is a
  convex polytope. Since both $P$ and the hyperplanes $H$ are
  independent of $\om$, also the polytopes $P(\ga)$ do not depend
    on $\om$.  By \cite[Corollary 5.8]{heinzner-schuetzdeller}
  \begin{gather*}
    \mup(Y)\cap \lia = \bigcup_{\ga \in F(\om)}P(\ga),
  \end{gather*}
  where $F(\om) \subset \Gamma$ is some subset of $C(P_0)$.  One
    can join $\om$ to $\omt$ continuously, e.g. by $\om_t: = \om + t
    dd^c f$.  Then $\mut_t : = \mu - t d^c f ( \cd _Z)$ also depends
    continuously on $t$. So $P(\ga)\subset \mup(Y)\cap \lia $ if and
    only if $P(\ga)\subset \mu_{t,\liep}(Y)\cap \lia $. Therefore
    $F(\om_t)$ is constant and the same is true of $\mup(Y)\cap
    \lia$. This implies $\mup(Y)=K(\mup(Y)\cap \lia)$. Hence $\mut(Y)
    = \mutp(Y)$.
\end{proof}

\begin{cor} \label{image-Z} Assume that $Z$ is connected and let
    $\om$ and $\omt$ be two cohomologous \Keler forms with momentum
    maps $\mu$ and $\mut$ respectively as in Lemma \ref{defmut}. Then
  $\mut$ is the unique momentum map such that $\mu(Z) = \mut(Z)$.
\end{cor}
\begin{proof}
  Since two momentum maps with respect to $\omt$ differ by addition of
  an element of the center of $\liu$, it is clear that there is at
  most one such map with the image equal to $\mu(Z)$.  To complete the
  proof it is therefore enough to check that $\mut(Z)=\mu(Z)$. This is
  a special case of the previous theorem.
\end{proof}

\begin{teo} \label{strato} Let $\om$ and $\omt$ be two cohomologous
  \Keler forms on $Z$, with momentum maps $\mu$ and $\mut$
  respectively as in Lemma \ref{defmut}.
  % ra \liep$ is proper.
  Then the set $\mathcal{B}_\liep^+$ is the same for both momentum
  maps and the two stratifications of $X$ coincide.
\end{teo}
\begin{proof}
  % It follows from \eqref{eq:defmut} and the definition of $\mup$
  % that
  % $\sx \tmu_\liep, \beta \xs = \sx \mup , \beta \xs + df(\beta_Z)
  % /2$.
  % Define $\nu : Z\ra \liep$ by the formula $\sx \nu(x), \beta \xs =
  % df(\beta_Z)(x)$, $x\in Z, \beta \in \liep$.  Since $Z$ is compact
  % and $\nu$ is continuous, there is $R>0$ such that $||\nu(x)||\leq
  % R$
  % for any $x\in Z$.  If $K\subset \liep$ is compact and $x\in
  % (\tmu_\liep\restr{X})\meno(K)$, then $\mup(x) = \tmu_\liep(x) -
  % \nu(x) \in Q$, where $Q$ denotes the set of points in $\liep$ with
  % distance at most $R$ from $K$.  We have proved that $
  % (\tmu_\liep\restr{X})\meno(K) \subset (\mup\restr{X})\meno(Q)$.
  % Since $Q$ is compact and $\mup\restr{X}$ is proper,
  % $(\mup\restr{X})\meno(Q)$ is compact. Since
  % $(\tmu_\liep\restr{X})\meno(K)$ is closed, it is also
  % compact. Therefore $\tmu_\liep\restr{X}$ is proper. This proves
  % the
  % first assertion.  To prove the second recall that
  By \cite[Corollary 7.6]{heinzner-schwarz-stoetzel}
  \begin{gather}\label{BB}
    \mathcal{B}_\liep = \{\beta \in \liep: \text{ there exists } x\in
    X: \frac{ ||\beta||^2}{2} = \inf_{G\cd x} \eta_\liep \text { and }
    \beta \in \mup (\overline{G\cd x}) \}.
  \end{gather}
  Moreover for $\beta \in \mathcal{B}_\liep$
  \begin{gather}\label{SB}
    S_\beta = \{ x\in X: \frac{||\beta||^2}{2} = \inf_{G\cd x}
    \eta_\liep \text { and } \beta \in \mup (\overline{G\cd x})\}.
  \end{gather}
  For any point $x\in X$, the set $\overline{G\cd x} $ is closed and
  $G$-invariant. Hence by Theorem \ref{image} $\mup( \overline{G\cd
    x}) = \mutp( \overline{G\cd x})$.  From this it follows that
  $\inf_{G\cd x} \eta_\liep = \inf_{G\cd x} \tilde{\eta}_\liep$, where
  $ \tilde{\eta}_\liep:= ||\mu_\liep||^2/2$.  The result follows from
  \eqref{BB} and \eqref{SB}.
\end{proof}

% \changed{The following convexity property of the \emph{gradient}
% momentum map was proven in \cite{heinzner-schuetzdeller}.}
% \begin{teo}[Heinzner-Sch\"utzdeller] \label{HS} Let $V$ be a unitary
%   vector space and let $\om$ be the Fubini-Study metric on
%   $\PP(V)$. Let $U$ be a compact connected semisimple Lie group
%   acting unitarily on $V$, let $G$ be a real form of $U^\C$ and let
%   $Y$ be a closed $G$-invariant real-semialgebraic subset of
%   $\PP(V)$ whose real-algebraic Zariski closure is irreducible.  Let
%   $\lia \subset \liep$ be a maximal subalgebra and let $\lia_+$ be a
%   closed Weyl chamber in $\lia$. Then $A(Y)_+:=\mup(Y)\cap\lia_+ $
%   is convex.
% \end{teo}
%
From the above we obtain the following generalization.
\begin{cor} \label{convexity} If $Z$ is a complex projective manifold,
  $U$ is a compact connected semisimple Lie group acting on $Z$, $\om$
  is a $U$-invariant Hodge metric and $Y\subset Z$ is a closed
  $G$-invariant real semi-algebraic subset whose real algebraic
  Zariski closure is irreducible, then $A(Y)_+$ is convex.  Moreover
  if $G$ is semisimple, then $X(\mu)$ is dense (if it is nonempty).
\end{cor}

\begin{proof}
  By assumption there is a very ample line bundle $L\ra Z$ such that
  $[\om]=2\pi \chern_1(L) / m$ for an interger $m>0$. Let $\om_{FS}$
  be a $U$-invariant Fubini-Study metric on $\PP(H^0(Z,L)^*)$. Let
  $\mu_{FS}$ be the moment map with respect to $\om_{FS}\restr{Z}$.
  In \cite{heinzner-schuetzdeller} the convexity theorem has been
  proved for $\mu_{FS}$.  A rescaling in the symplectic form yields a
  corresponding rescaling in the momentum map.  Therefore the
  convexity theorem also holds for the momentum map $\mut$ relative to
  the symplectic form $\omt:=\om_{FS}/m$.  So it holds also for $\mu$,
  since $\mu_\liep(Y)=\mut_\liep(Y)$ by Theorem \ref{image}.  The
  proof of the last statement is similar: see
  \cite{heinzner-schuetzdeller} and Corollary \ref{image-Z}.
\end{proof}
\begin{cor}
  Under the same assumptions, any local minimum of $|\mup|^2$ is a
  global minimum.
\end{cor}
\begin{proof}
  This follows since $| \mup |^2$ is $K$-invariant and $\mu(Z)_+$ is a
  convex subset of $\mathfrak a_+$.
\end{proof}

\begin{cor}
  If $\om $ and $\om'$ are cohomologous \Keler forms on $Z$ with
  momentum maps $\mu$ and $\tmu$ as in Lemma \ref{defmut}, then
  $X(\mu) = X(\tmu)$.
\end{cor}
\begin{proof}
  It is enough to observe that $X(\mu) = S_0$.
\end{proof}

\end{document}